\documentclass[11pt]{article}

\usepackage[english]{babel}
\usepackage[T1]{fontenc}
\usepackage{amsfonts}
\usepackage{amsmath}
\usepackage{amssymb}
\usepackage{amsthm}
\usepackage{stmaryrd}
\usepackage{dsfont}
\usepackage{graphicx}
\usepackage{hyperref}
\usepackage{titlesec}
\usepackage{mathabx}
\usepackage{enumitem,tocloft}
\usepackage[top=2.5cm, bottom=2.5cm, left=3cm, right=3cm]{geometry}
\usepackage{xcolor}
\usepackage{fourier-orns}
\usepackage[utf8]{inputenc}
\usepackage{bm}
\usepackage{framed}
\usepackage{supertabular}
\hypersetup{linktocpage,breaklinks=true}

\newcommand{\Z}{\mathbb{Z}}
\newcommand{\N}{\mathbb{N}}
\newcommand{\R}{\mathbb{R}}

\newcommand{\ld}{\mathrm{L}}
\newcommand{\esp}{(X,\mu)}

\newcommand{\espom}{(\Omega,\mu)}

\newcommand{\G}{\Gamma}
\newcommand{\g}{\gamma}
\newcommand{\La}{\Lambda}
\newcommand{\la}{\lambda}
\newcommand{\cogl}{c_{\G,\La}}
\newcommand{\cogldef}{\cogl\colon\G\times\XL\to\La}
\newcommand{\colg}{c_{\La,\G}}
\newcommand{\colgdef}{\colg\colon\La\times\XG\to\G}
\newcommand{\XG}{X_{\G}}
\newcommand{\XL}{X_{\La}}
\newcommand{\jg}{j_{1,\G}}
\newcommand{\jl}{j_{1,\La}}
\newcommand{\fg}{h_{\G}}
\newcommand{\fl}{h_{\La}}

\newtheorem{theorem}{Theorem}[section]
\newtheorem*{theorem*}{Theorem}
\newtheorem{corollary}[theorem]{Corollary}
\newtheorem*{corollary*}{Corollary}

\newtheorem{lemma}[theorem]{Lemma}

\newtheorem*{claim*}{Claim}
\newtheorem*{fact*}{Fact}

\newtheorem{theoremletter}{Theorem}

\newtheorem{corollaryletter}[theoremletter]{Corollary}

\theoremstyle{definition}
\newtheorem{definition}[theorem]{Definition}
\newtheorem{remark}[theorem]{Remark}
\newtheorem{question}[theorem]{Question}

\title{On the absence of quantitatively critical measure equivalence couplings}
\author{Corentin Correia}
\date{March 06, 2024}

\begin{document}
	
	\maketitle
	
	\begin{abstract}
		Given a measure equivalence coupling between two finitely generated groups, Delabie, Koivisto, Le Maître and Tessera have found explicit upper bounds on how integrable the associated cocycles can be. These bounds are optimal in many cases but the integrability of the cocycles with respect to these critical thresholds remained unclear. For instance, a cocycle from $\Z^{k+\ell}$ to $\Z^{k}$ can be $\ld^p$ for all $p<\frac{k}{k+\ell}$ but not for $p>\frac{k}{k+\ell}$, and the case $p=\frac{k}{k+\ell}$ was an open question which we answer by the negative. Our main result actually yields much more examples where the integrability threshold given by Delabie-Koivisto-Le Maître-Tessera Theorems cannot be reached.
	\end{abstract}
	
	{
		\small	
		\noindent\textbf{{Keywords:}} Quantitative measure equivalence, amenable groups, isoperimetric profile. 
	}
	
	\smallskip
	
	{
		\small	
		\noindent\textbf{{MSC-classification:}}	
		Primary 37A20; Secondary 20F65.
	}
	
	\section{Introduction}
	
	Measure equivalence is an equivalence relation on countable groups introduced by Gromov as a measured analogue of quasi-isometry. A first example of measure equivalent groups is given by two lattices in the same locally compact group.\par
	Another source of examples is provided by orbit equivalence. Two groups $\G$ and $\La$ are \textit{orbit equivalent} if there exist two free probability measure-preserving $\G$- and $\La$-actions $\alpha_{\G}$ and $\alpha_{\La}$ on a standard probability space $\esp$, having the same orbits. This yields measurable functions $\cogl\colon\G\times X\to\La$ and $\colg\colon\La\times X\to\G$ describing the distortions on the orbits, called the \textit{cocycles} and defined almost everywhere by the equations
	$$\alpha_{\G}(\g)x=\alpha_{\La}(\cogl(\g,x))x\ \text{and}\ \alpha_{\La}(\la)x=\alpha_{\G}(\colg(\la,x))x.$$
	More generally, the notion of measure equivalence also yields cocycles $\cogldef$ and $\colgdef$, where $(\XG,\mu_{\XG})$ and $(\XL,\mu_{\XL})$ are probability spaces arising from the measure equivalence coupling between the groups (see Section~\ref{secqme}).\par
	When the two groups are finitely generated, a stronger notion called $\ld^1$\textit{ measure equivalence} can be defined. It requires that the measurable functions $\left |\cogl(\g,\cdot)\right |_{S_{\La}}$ and $\left |\colg(\la,\cdot)\right |_{S_{\G}}$ are integrable for every $\g\in\G$ and $\la\in\La$, where ${|.|_{S_{\G}}}$ and ${|.|_{S_{\La}}}$ respectively denote the \textit{word-length metrics} with respect to some finite generating sets $S_{\G}$ and $S_{\La}$ of the groups. This definition does not depend on the choice of $S_{\G}$ and $S_{\La}$ and we simply say that $\cogl$ and $\colg$ are integrable. Many rigidity results have been uncovered in this context (see e.g.~\cite{baderIntegrableMeasureEquivalence2013} and~\cite{austinIntegrableMeasureEquivalence2016}). Most of the time, these results tell us that $\ld^1$ measure equivalence captures the geometry of the groups, in contrast to Ornstein-Weiss Theorem~\cite{ornsteinErgodicTheoryAmenable1980} which states that all infinite countable amenable groups are measure equivalent.\par
	
	To get finer rigidity results among finitely generated groups, Delabie, Koivisto, Le Maître and Tessera~\cite{delabieQuantitativeMeasureEquivalence2022} introduced more general quantitative restrictions on the cocycles. Given positive real numbers $p$ and $q$, we say that two finitely generated groups $\G$ and $\La$ are $(\ld^p,\ld^q)$ \textit{measure equivalent} (resp. $(\ld^p,\ld^q)$ \textit{orbit equivalent}) if there exists a measure equivalence (resp. an orbit equivalence) between them and the associated cocycles $\cogl$ and $\colg$ are respectively $\ld^p$ and $\ld^q$, i.e. the real-valued measurable functions $\left |\cogl(\g,\cdot)\right |_{S_{\La}}$ and $\left |\colg(\la,\cdot)\right |_{S_{\G}}$ are respectively $\ld^p$ and $\ld^q$ for every $\g\in\G$ and $\la\in\La$. We also replace $\ld^p$ or $\ld^q$ by $\ld^0$ when no requirement is made on the corresponding cocycle.\par
	We can also define $(\varphi,\psi)$-integrability measure equivalence (resp. orbit equivalence) for non-decreasing maps $\varphi,\psi\colon\R_+\to\R_+$ (see Definition~\ref{quantitative}). In particular, $\ld^p$ means that we consider the map $x\mapsto x^p$.\newline
	
	In the case of the groups $\Z^d$, for $d\geq 1$, Delabie, Koivisto, Le Maître and Tessera prove that there is no $(\ld^p,\ld^0)$ measure equivalence coupling from $\Z^{k+\ell}$ to $\Z^k$ for $p>\frac{k}{k+\ell}$ (\cite[Corollary 3.4]{delabieQuantitativeMeasureEquivalence2022}). On the other hand, they explicitly build a measure equivalence from $\Z^{k+\ell}$ to $\Z^k$ which is $(\ld^p,\ld^0)$ for every $p<\frac{k}{k+\ell}$ (\cite[Theorem~1.9]{delabieQuantitativeMeasureEquivalence2022}).\par
	The existence of a $(\ld^{\frac{k}{k+\ell}},\ld^0)$ measure equivalence coupling from $\Z^{k+\ell}$ to $\Z^k$ remained unclear (see also~\cite[Question 1.10]{delabieQuantitativeMeasureEquivalence2022}). Our contribution provides a negative answer to this question (see Corollary~\ref{cor4}), thus yielding the following complete description:
	
	\begin{theoremletter}[see Theorem~\ref{th3}]\label{corB}
		If $k$ and $\ell$ are positive integers, then there exists a $(\ld^p,\ld^0)$ measure equivalence coupling from $\Z^{k+\ell}$ to $\Z^k$ if and only if $p<\frac{k}{k+\ell}$.
	\end{theoremletter}
	
	The absence of measure equivalence coupling from $\Z^{k+\ell}$ to $\Z^k$ with the critical integrability $(\ld^{\frac{k}{k+\ell}},\ld^0)$ was the initial goal of the paper. As we will see later in this introduction, this is actually a particular case of more general statements (see Theorems~\ref{thA} and~\ref{thB}).\par
	Our result relies on the following key lemma (Lemma~\ref{lemma2} in the easier case $\Gamma=\Z^{k+\ell}$ and $\La=\Z^k$): given a measure equivalence coupling from $\Z^{k+\ell}$ and $\Z^k$, if a cocycle is $\varphi$-integrable, then it is $\psi$-integrable for another non-decreasing map $\psi\colon\R_+\to\R_+$ such that $\psi(x)=O(\varphi(x))$ does \emph{not} hold as $x$ goes to $+\infty$. If now we assume that the cocycle $c_{\Z^k,\Z^{k+\ell}}$ is $\varphi$-integrable where $\varphi(x)=x^{\frac{k}{k+\ell}}$, we can combine this with a more precise version of \cite[Corollary 3.4]{delabieQuantitativeMeasureEquivalence2022}: by their Theorem~3.1, we must have $\psi(x)=O(\varphi(x))$, a contradiction, thus proving our result. It is interesting to note that while the statement does not mention $\varphi$-integrability, its proof crucially uses it.\par
	This key lemma is a natural adaptation of the following elementary, yet fundamental fact.
	
	\begin{fact*}
		Let $(u_n)_{n\in\N}$ be a sequence of non-negative real numbers which is summable. Then there exists a sequence $(v_n)_{n\in\N}$ of non-negative real numbers which is summable and such that $u_n=o(v_n)$.
	\end{fact*}
	
	\begin{proof}[Proof of the fact]
		We can find an increasing sequence $(N_k)_{k\geq 1}$ of positive integers satisfying $N_1=0$ and $\sum_{n\geq N_k}^{+\infty}{u_n}\leq \frac{1}{k^3}$ for every $k\geq 2$.
		Then for every integer $n\geq 0$, we define $v_n\coloneq ku_n$ if $N_k\leq n<N_{k+1}$. We have
		$$\sum_{n=0}^{+\infty}{v_n}=\sum_{k=1}^{+\infty}{\sum_{n=N_k}^{N_{k+1}-1}{ku_n}}\leq\sum_{k=1}^{+\infty}{\frac{1}{k^2}}<+\infty$$
		and $u_n=o(v_n)$.
	\end{proof}

	Let us now present generalizations to other groups, using the isoperimetric profile (Theorem~\ref{thA}) and then the growth (Theorem~\ref{thB}). First, recall that given non-decreasing real-valued functions $f$ and $g$ defined on a neighborhood of $+\infty$, we say that $f$ is \textit{asymptotically less} than $g$, denoted by $f\preccurlyeq g$, if there exists a constant $C>0$ such that $f(x)=O\left (g(Cx)\right )$ as $x\to +\infty$. We say that $f$ is \textit{asymptotically equivalent} to $g$, denoted by $f\approx g$, if $f\preccurlyeq g$ and $f\succcurlyeq g$. The \textit{asymptotic behavior} of $f$ is its equivalence class modulo $\approx$.\par
	Given a finitely generated group $\G$, its \textit{isoperimetric profile} is a real-valued function $\jg$ defined on the set of positive integers and given modulo $\approx$ by the formula
	$$\jg(x)\approx\sup_{A\subset\G, |A|\leq x}{\frac{|A|}{|\partial A|}},$$
	where $\partial A\coloneq S_{\G}A\ \Delta\ A$ and $S_{\G}$ is a finite generating subset of $\G$. It has been computed for many groups, for instance $j_{1,\mathbb{Z}^d}(x)\approx x^{1/d}$ \cite{coulhonRandomWalksGeometry2000}, $j_{1,(\mathbb{Z}/2\mathbb{Z})\wr\Z}(x)\approx\log{x}$ \cite{erschlerIsoperimetricProfilesFinitely2003}, where $(\mathbb{Z}/2\mathbb{Z})\wr\Z$ is a lamplighter group (the definition is recalled in Section~\ref{secappwr}). Note that it is an unbounded function if and only if the group is amenable. It can thus be interpreted as a measurement of amenability: the faster it goes to infinity, the "more amenable" the group is. We refer the reader to \cite{delabieQuantitativeMeasureEquivalence2022} for more details on the isoperimetric profile and more generally the $\ell^p$-isoperimetric profile.\par
	Now we state the theorem of Delabie, Koivisto, Le Maître and Tessera on the behaviour of the isoperimetric profile under quantitative measure equivalence.
	
	\begin{theorem}[{\cite[Theorem~1.1]{delabieQuantitativeMeasureEquivalence2022}}]\label{th1}
		Let $\varphi\colon\R_+\to\R_+$ be a function such that $\varphi$ and $t\mapsto t/\varphi(t)$ are non-decreasing, let $\G$ and $\La$ be finitely generated groups. Assume that there exists a $(\varphi,\ld^0)$-integrable measure equivalence coupling from $\G$ to $\La$. Then their isoperimetric profiles satisfy the following asymptotic inequality:
		$$\varphi\circ\jl\preccurlyeq\jg.$$
	\end{theorem}
	
	If $\jl$ is injective, then $\varphi\circ\jl\preccurlyeq\jg$ means that there exists a constant $C>0$ such that the following holds as $x$ goes to $+\infty$:
	\begin{equation}\label{bound}
		\varphi(x)=O(\jg(C\jl^{-1}(x))),
	\end{equation}
	so Theorem~\ref{th1} provides upper bounds $x\mapsto\jg(C\jl^{-1}(x))$ for $C>0$. In order to generalize our first contribution ("there is no $(\ld^{\frac{k}{k+\ell}},\ld^0)$ measure equivalence coupling from $\Z^{k+\ell}$ from $\Z^k$") to other groups, we must pay attention to a few obstacles which do not appear in the case $\G=\Z^{k+\ell}$ and $\La=\Z^{k}$.
	\begin{itemize}
		\item The isoperimetric profile of a finitely generated group $\La$ is not necessarily injective, so \eqref{bound} is not well-defined in full generality. But when studying this function, we only take into account its asymptotic behaviour. Moreover, we will check that it suffices to consider an injective function $h_{\La}$ with the same asymptotic behavior (the existence of such a map is granted by Remark~\ref{isoinj}).
		\item Given two different positive constants $C$ and $C'$, we do not know if the functions $\jg(C\jl^{-1}(.))$ and $\jg(C'\jl^{-1}(.))$ have the same asymptotic behavior, so Theorem~\ref{th1} does not provide a precise upper bound of $\varphi$ \textit{a priori}. This is the reason why we will assume that the isoperimetric profile of $\G$ satisfies $\jg(Cx)=O(\jg(x))$ for every $C>0$. For other technical reasons arising from the existence of a constant in the definition of "$\varphi$-integrability" (see Definition~\ref{quantitative}), we will also require this hypothesis on $\jg\circ\jl^{-1}$. These requirements motivate Assumptions~\eqref{hyp2} and~\eqref{hyp3} in Theorem~\ref{thA} below.
		\item In Lemma~\ref{lemma2}, where we build a new map $\psi$ from the original one $\varphi\coloneq \jg\circ \jl^{-1}$ (for the case $\G=\Z^{k+\ell}$ and $\La=\Z^{k}$, see the paragraph after the proof of the elementary fact), we need $\varphi$ to be sublinear\footnote{This is necessary to assume that $\jg\circ\jl^{-1}$ is sublinear. Indeed, we cannot apply the same strategy in the case $\G=\Z$ and $\La=\Z^{2}$, since Escalier and Joseph have built a measure equivalence coupling from $\Z$ to $\Z^{2}$ which is $(\ld^{\infty},\ld^{p})$ for every $p<\frac{1}{2}$ (not yet published work).}, hence Assumption~\eqref{hyp1} in Theorem~\ref{thA}.
	\end{itemize}
	
	Hence, a first generalization is the following.
	
	\begin{theoremletter}\label{thA}
		Let $\G$ and $\La$ be finitely generated groups. Assume that there exist a non-decreasing function $\fg$ and an increasing function $\fl$ satisfying $\fg\approx\jg$, $\fl\approx\jl$ and the following assumptions as $x\to +\infty$:
		\begin{equation}\label{hyp1}
			\fg(x)=o\left (\fl(x)\right ),
		\end{equation}
		\begin{equation}\label{hyp2}
			\forall C>0,\ \fg(Cx)=O\left (\fg(x)\right ),
		\end{equation}
		\begin{equation}\label{hyp3}
			\forall C>0,\ \fg\circ \fl^{-1}(Cx)=O\left (\fg\circ \fl^{-1}(x)\right ).
		\end{equation}
		Then there is no $(\fg\circ \fl^{-1},\ld^0)$-integrable measure equivalence coupling from $\G$ to $\La$.
	\end{theoremletter}
	
	\begin{remark}\label{isoinj}
		The isoperimetric profile of a finitely generated group $\G$ is always asymptotically equivalent to an increasing function $\fg$. For instance, if $\jg$ satisfies
		$$0<\jg(n-1)<\jg(n)=\ldots=\jg(n+k-1)<\jg(n+k)$$
		for some positive integers $n$ and $k$, then we can set
		$$\fg(n+i)\coloneq \frac{k-i}{k}\jg(n)+\frac{i}{k}\min{(\jg(n+k),2\jg(n))}$$
		for every $i\in\{0,\ldots,k-1\}$. We do not provide the details.\par
		It is straightforward to check that the equivalence relation $\approx$ preserves Assumption~\eqref{hyp2} for a non necessarily injective function. Moreover satisfying Assumptions~\eqref{hyp1} and~\eqref{hyp2} is also preserved under this equivalence relation, as well as satisfying Assumptions~\eqref{hyp1},~\eqref{hyp2} and~\eqref{hyp3} when the inverse of one of the functions is well-defined.
	\end{remark}
	
	\begin{question}\label{q3}
		Does the isoperimetric profile of a finitely generated group always satisfy Assumption~\eqref{hyp2}? In the case $\jg(x)=o(\jl(x))$, does there always exist a pair $(\fg,\fl)$ of functions satisfying the assumptions of Theorem~\ref{thA}?
	\end{question}
	
	The following corollary allows us to answer a question of Delabie, Koivisto, Le Maître and Tessera (see~\cite[Question 1.2]{delabieQuantitativeMeasureEquivalence2022}) by the negative for many of finitely generated group~$\G$.
	
	\begin{corollaryletter}[see Corollary~\ref{cor3}]
		Let $\G$ be a finitely generated group which is not virtually cyclic. Assume that its isoperimetric profile $\jg$ satisfies
		\begin{equation}\label{hyp8}
			\forall C>0,\ \jg(Cx)=O\left (\jg(x)\right )\text{ as }x\to +\infty.
		\end{equation}
		Then there is no $(\jg,\ld^0)$-integrable measure equivalence coupling from $\G$ to $\Z$.
	\end{corollaryletter}
	
	Given an increasing function satisfying a mild regularity condition, Brieussel and Zheng~\cite{brieusselSpeedRandomWalks2021} build a group whose isoperimetric profile is asymptotically equivalent to this function. It turns out that this regularity condition implies our condition~\eqref{hyp8} (see Section~\ref{secappZ}). Moreover, if $\G$ is such a group\footnote{We call it a Brieussel-Zheng group, although their construction is more general.}, it follows from the work of Escalier~\cite{escalierBuildingPrescribedQuantitative2024} that there exists an orbit equivalence from $\G$ to $\Z$ which is almost $(\jg,\ld^0)$-integrable, thus providing a complete description similar to Theorem~\ref{corB} (see Theorem~\ref{BrieusselZhengEscalier}).\par
	Explicit constructions of orbit equivalences in~\cite{delabieQuantitativeMeasureEquivalence2022} show that the upper bound given in~Theorem~\ref{th1} is sharp for other groups than $\Z^d$, such as lamplighter groups or iterated wreath products. The existence of a measure equivalence coupling with this critical threshold remained unclear and our Theorem~\ref{thA} enables us to answer by the negative. We refer the reader to Theorems~\ref{th4},~\ref{th5},~\ref{th6} and~\ref{th7} for precise statements.\newline
	
	Another rigidity result in~\cite{delabieQuantitativeMeasureEquivalence2022} deals with the notion of volume growth. Given a finitely generated group $\G$ and finite generating set $S_{\G}$ of $\G$, we define
	$$V_{\G}(n)\coloneq \left |\left \{\gamma_1\ldots\gamma_n\mid\gamma_1,\ldots,\gamma_n\in S_{\G}\cup (S_{\G})^{-1}\cup\{e_\G\}\right \}\right |$$
	for every positive integer $n$, where $e_\G$ denotes the identity element of $\G$. As for the isoperimetric profile, we extend $V_{\G}$ to a continuous and non-decreasing function. The \textit{volume growth} of $\G$ is the asymptotic behavior of $V_{\G}$, it does not depend on the choice of $S_{\G}$, nor does its extension to $\R_+$. We say that $\G$ has \textit{polynomial growth of degree }$d$ if $V_{\G}(x)\approx x^d$. Finally, note that the volume growth is increasing but the isoperimetric profile may fail to be injective.
	
	\begin{theorem}[{\cite[Theorem~3.1]{delabieQuantitativeMeasureEquivalence2022}}]\label{th2}
		Let $\varphi$ be an increasing, subadditive function such that $\varphi(0)=0$, let $\G$ and $\La$ be finitely generated groups. Assume that there exists a $(\varphi,\ld^0)$-integrable measure equivalence coupling from $\G$ to $\La$. Then
		$$V_{\G}\preccurlyeq V_{\La}\circ\varphi^{-1},$$
		where $\varphi^{-1}$ denotes the inverse function of $\varphi$.
	\end{theorem}
	
	With the same strategy as Theorem~\ref{thA}, we get the following statement.
	
	\begin{theoremletter}\label{thB}
		Let $\G$ and $\La$ be finitely generated groups. Assume that there exist two increasing functions $\fg$ and $\fl$ satisfying $\fg\approx V_{\G}$, $\fl\approx V_{\La}$ and the following properties as $x\to +\infty$:
		\begin{equation}\label{hyp4}
			\fg^{-1}(x)=o\left (\fl^{-1}(x)\right ),
		\end{equation}
		\begin{equation}\label{hyp5}
			\forall C>0,\ \fg^{-1}(Cx)=O\left (\fg^{-1}(x)\right ),
		\end{equation}
		\begin{equation}\label{hyp6}
			\forall C>0,\ \fg^{-1}\circ\fl(Cx)=O\left (\fg^{-1}\circ\fl(x)\right ).
		\end{equation}
		Then there is no $(\fg^{-1}\circ\fl,\ld^0)$-integrable measure equivalence coupling from $\G$ to $\La$.
	\end{theoremletter}
	
	We will prove Theorems~\ref{thA} and~\ref{thB} in Section~\ref{secproof} and give the main applications in Section~\ref{secapp}.\newline
	
	More general statements of Delabie, Koivisto, Le Maître and Tessera deal with asymmetric weakenings of measure equivalence coupling: measure subgroup, quotient and sub-quotient couplings. We can still apply our ideas to these generalizations.\par
	Theorems~\ref{thA} and~\ref{thB} still hold in the context of quantitative orbit equivalence, since the existence of a $(\varphi,\psi)$-integrable orbit equivalence from $\G$ to $\La$ is equivalent to the existence of a $(\varphi,\psi)$-integrable measure equivalence coupling with equal fundamental domains.
	
	\paragraph{Acknowledgements.}
	
	I am very grateful to my advisors François Le Maître and Romain Tessera for their support and valuable advice. I also wish to thank Amandine Escalier and Matthieu Joseph for fruitful discussions about the construction of measure equivalence couplings between $\Z^2$ and $\Z$. Finally, I thank Vincent Dumoncel and Fabien Hoareau for their useful comments on the paper.
	
	\section{Quantitative measure equivalence}\label{secqme}
	
	The groups $\G$ and $\La$ are always assumed to be finitely generated. By a \textit{smooth} action of a countable group $\G$, we mean a measure-preserving $\G$-action on a standard measured space $\espom$ which admits a fundamental domain, namely a Borel subset $\XG$ of $\Omega$ that intersects every $\G$-orbit exactly once.
	
	\begin{definition}\label{defme}
		A \textbf{measure equivalence coupling} between $\G$ and $\La$ is a quadruple $(\Omega,\XG,\XL,\mu)$ where $\espom$ is a standard Borel measure space equipped with commuting measure-preserving smooth $\G$- and $\La$-actions such that
		\begin{enumerate}
			\item both the $\G$- and $\La$-actions are free;
			\item $\XG$ (resp. $\XL$) is a fixed fundamental domain for the $\G$-action (resp. for the $\La$-action);
			\item $\XG$ and $\XL$ have finite measures.
		\end{enumerate}
		We will always use the notations $\gamma\ast x$ and $\lambda\ast x$ (with $\g\in\G$, $\la\in\La$, $x\in\Omega$) for these smooth actions on $\Omega$. The notations $\g\cdot x$ and $\la\cdot x$ refers to the induced actions that we now define, as well as the cocycles.
	\end{definition}
	
	\begin{definition}
		A measure equivalence coupling $(\Omega,\XG,\XL,\mu)$ between $\G$ and $\La$ induces a finite measure-preserving $\G$-action on $(\XL,\mu_{\XL})$ in the following way: for every $\g\in\G$ and every $x\in\XL$, $\g\cdot x\in\XL$ is defined by the identity
		$$(\La\ast\g\ast x)\cap\XL=\{\g\cdot x\},$$
		it is unique since $\XL$ is a fundamental domain for the smooth $\La$-action.\par
		This also yields a \textbf{cocycle} $\cogldef$ uniquely (by freeness) defined by
		$$\cogl(\g,x)\ast\g\ast x=\gamma\cdot x,$$
		or equivalently $\cogl(\g,x)\ast\g\ast x\in\XL$, for almost every $x\in\XL$ and every $\g\in\G$. We similarly define a finite measure-preserving $\La$-action on $(\XG,\mu_{\XG})$ and the associated cocycle $\colgdef$.
	\end{definition}
	
	\begin{remark}
		The cocycle $\cogldef$ satisfies the cocycle identity
		$$\forall\g_1,\g_2\in\G,\ \forall x\in\XL,\ \cogl(\g_1\g_2,x)=\cogl(\g_1,\g_2\cdot x)\cogl(\g_2,x).$$
	\end{remark}
	
	\begin{definition}[Delabie, Koivisto, Le Maître and Tessera~\cite{delabieQuantitativeMeasureEquivalence2022}]\label{quantitative}
		Let $\varphi\colon\R_+\to\R_+$ be a non-decreasing map. Given a measure equivalence coupling between $\G$ and $\La$, we say that the cocycle $\cogldef$ is $\varphi$\textbf{-integrable} if for every $\g\in\G$, there exists $c_{\g}>0$ such that
		$$\int_{\XL}{\varphi\left (\frac{{|\cogl(\g,x)|}_{S_{\La}}}{c_{\g}}\right )\mathrm{d}\mu_{\XL}(x)}<+\infty$$
		where $S_{\La}$ is a finite generating set of $\La$ and for every $\la$, ${|\la|_{S_{\La}}}$ denotes its \textbf{word-length metric} with respect to $S_{\La}$, defined by
		$${|\la|_{S_{\La}}}\coloneq \min\{n\geq 0\mid\exists\la_1,\ldots,\la_n\in S_{\La}\cup (S_{\La})^{-1}\cup\{e_{\Lambda}\},\la=\la_1\ldots\la_n\}.$$
		We define $\varphi$-integrability for $\colg$ in a similar way.
	\end{definition}
	
	\begin{remark}
		Defining $\varphi$-integrability for the cocycle $\cogl$ with the use of constants $c_{\gamma}$ is necessary because we need the following properties:
		\begin{itemize}
			\item this notion of $\varphi$-integrability does not depend on the choice of the finite generating set of $\La$, since for any finitely generated sets $S_{\La},S'_{\La}$, there exists a constant $C>0$ such that
			$$\frac{1}{C}{\left |\la\right |}_{S'_{\La}}\leq {\left |\la\right |}_{S_{\La}}\leq C{\left |\la\right |}_{S'_{\La}}$$
			for every $\la\in\La$;
			\item if $\varphi\approx\psi$, then $\varphi$-integrability and $\psi$-integrability are equivalent notions;
			\item to prove that the cocycle $\cogldef$ is $\varphi$-integrable, it suffices to check the finiteness of
			$$\int_{\XL}{\varphi\left (\frac{{\left |\cogl(\g,x)\right |}_{S_{\La}}}{c_{\g}}\right )\mathrm{d}\mu_{\XL}(x)}$$
			for every element $\g$ in a finite generating set of $\G$. This follows from~\cite[Proposition~2.22]{delabieQuantitativeMeasureEquivalence2022}.
		\end{itemize}
	\end{remark}
	
	\begin{definition}[Delabie, Koivisto, Le Maître and Tessera~\cite{delabieQuantitativeMeasureEquivalence2022}]
		A measure equivalence coupling $(\Omega,\XG,\XL,\mu)$ between the groups $\G$ and $\La$ is a $(\varphi,\psi)$\textbf{-integrable measure equivalence coupling} from $\G$ to $\La$ if $\cogldef$ is $\varphi$-integrable and $\colgdef$ is $\psi$-integrable.\par
		For $p>0$, we write $\ld^p$ instead of $\varphi$ or $\psi$ if we consider the map $t\mapsto t^p$, and we write $\ld^0$ when no requirement is made on the cocycle. For example, the measure equivalence coupling is $(\varphi,\ld^p$)-integrable if $\cogl$ is $\varphi$-integrable and $\colg$ is in $\ld^p(\XL,\mu_{\XL})$; it is $(\ld^p,\ld^0$)-integrable if $\cogl$ is $\ld^p\esp$. Finally, a measure equivalence coupling is $\varphi$\textbf{-integrable} if it is $(\varphi,\varphi)$-integrable.
	\end{definition}
	
	Note that a $(\varphi,\psi)$-integrable measure equivalence coupling from $\G$ to $\La$ is a $(\psi,\varphi)$-integrable measure equivalence coupling from $\La$ to $\G$.
	
	\section{Proof of the main results}\label{secproof}
	
	We now prove Theorems~\ref{thA} and~\ref{thB}. The key result is Lemma~\ref{lemma2}, which uses Lemma~\ref{lemma1}.
	
	\begin{lemma}\label{lemma1}
		Let $x\in\R$ and $\theta\colon [x,+\infty )\to\R$ be a continuous sublinear function. If $y$ is a real number satisfying $y<\theta(t)$ for every $t\in [x,+\infty)$, then the set
		$$E(x,y,\theta)\coloneq \left\{t> x\mid\forall s\in [x,t], \theta(s)\geq\frac{\theta(t)-y}{t-x}(s-x)+y\right\}$$
		is not bounded above.
	\end{lemma}
	
	\begin{proof}[Proof of Lemma~\ref{lemma1}]
		Let us consider the continuous maps $a\ \colon t\in (x,+\infty)\mapsto\R$ and $m\ \colon t\in (x,+\infty)\mapsto\R$ defined by
		$$a(t)=\frac{\theta(t)-y}{t-x}\ \text{and}\ m(t)=\min_{s\in (x,t]}{a(s)}.$$
		Note that the set $E(x,y,\theta)$ is equal to $\{t>x\mid m(t)=a(t)\}$. Let us also define the set
		$$E'\coloneq \{t\in (x,+\infty)\mid\forall s\in (x,t),\ m(s)>m(t)\}.$$
		By the assumptions, the non-increasing map $m$ satisfies the following properties:
		\begin{itemize}
			\item $m(t)>0$ for every $t\in (x,+\infty)$;
			\item $m(t)\underset{t\to +\infty}{\longrightarrow}0$;
			\item if $t$ is in $E'$, then we have $m(t)=a(t)$.
		\end{itemize}
		Therefore the set $E'$ is not bounded above and is included in $E(x,y,\theta)$.
	\end{proof}
	
	\begin{lemma}\label{lemma2}
		Let $\varphi\colon\R_+\to\R_+$ be a continuous, sublinear and increasing function. Given an integer $\ell\geq 1$ and a probability space $\esp$, let $f_1,\ldots,f_{\ell}\colon X\to\N$ be measurable maps satisfying
		$$\int_{X}{\varphi(f_i(x))\mathrm{d}\mu(x)}<+\infty$$
		for every $i\in\{1,\ldots,\ell\}$. Then there exists a subadditive map $\psi\colon\R_+\to\R_+$ such that $\psi(0)=0$, $\psi$ and $t\mapsto t/\psi(t)$ are non-decreasing, and
		\begin{enumerate}
			\item $\varphi(x_k)=o(\psi(x_k))$ for some increasing sequence $(x_k)_{k\geq 0}$ of non-negative real numbers tending to $+\infty$;
			\item for every $i\in\{1,\ldots,\ell\}$,
			$$\int_{X}{\psi(f_i(x))\mathrm{d}\mu(x)}<+\infty.$$
		\end{enumerate}
	\end{lemma}
	
	\begin{proof}[Proof of Lemma~\ref{lemma2}]
		For every $n\geq 0$ and every $i\in\{1,\ldots,\ell\}$, let us define the non-negative real number $u^{(i)}_n\coloneq \varphi(n)\mu(\{f_i=n\})$. For every $i\in\{1,\ldots,\ell\}$, the sequence $(u^{(i)}_n)_{n\geq 0}$ is summable since
		$$\sum_{n=0}^{+\infty}{u^{(i)}_n}=\sum_{n=0}^{+\infty}{\varphi(n)\mu(\{f_i=n\})}=\int_X{\varphi(f_i(x))\mathrm{d}\mu(x)}<\infty.$$
		Let $(N_k)_{k\geq 1}$ be an increasing sequence of positive integers satisfying $N_1=0$ and
		$$\forall k\geq 2,\ \forall i\in\{1,\ldots,\ell\},\ \sum_{n=N_k}^{+\infty}{u^{(i)}_n}\leq \frac{1}{k^3}.$$
		Then for every integer $n\geq 1$, we define $K_n\coloneq k$ if $N_k\leq n<N_{k+1}$. The sequence $(K_n)_{n\geq 1}$ tends to $+\infty$ and the sequences $(K_nu^{(i)}_n)_{n\geq 1}$ are summable (see the proof of the fact in the introduction).\par
		We inductively build an increasing sequence $(x_k)_{k\geq 0}$ of integers satisfying $x_0=0$ and $x_k\geq N_{k+1}$ for every $k\geq 1$, a decreasing sequence $(a_k)_{k\geq 0}$ of positive real numbers, a sequence $(b_k)_{k\geq 0}$ of non-negative real numbers satisfying $b_0=0$, and a continuous piecewise linear map $\psi\colon\R_+\to\R_+$ satisfying the following properties:
		\begin{itemize}
			\item for every $k\geq 0$, for every $t\in [x_k,x_{k+1}]$, $\psi(t)=b_k+a_kt$ and $\psi(t)\leq (k+1)\varphi(t)$;
			\item for every $k\geq 0$, $\psi(x_k)=k\varphi(x_k).$
		\end{itemize}
		Let us set $x_0\coloneq 0$, $x_1\coloneq N_2$, $a_0=\varphi(N_2)/N_2$, $b_0=0$ and for every $t\in [0,N_2]$,
		$$\psi(t)\coloneq \frac{\varphi(N_2)}{N_2}t.$$
		Given an integer $k\geq 2$, assume that we have already defined $0=x_0<x_1<\ldots <x_{k-1}$, $a_0>\ldots >a_{k-2}$, $b_0,\ldots,b_{k-2}$ and the map $\psi$ on $[0,x_{k-1}]$. By the assumptions on $\varphi$ and since
		$$\psi(x_{k-1})=(k-1)\varphi(x_{k-1})<k\varphi(x_{k-1}),$$
		we can apply Lemma~\ref{lemma1} to $x\coloneq x_{k-1}$, $y\coloneq \psi(x_{k-1})$, $\theta\coloneq k\times\varphi$. We choose $x_k\in E(x_{k-1},\psi(x_{k-1}),k\times\varphi)$ sufficiently large so that
		\begin{itemize}
			\item $x_k\geq N_{k+1}$;
			\item $\displaystyle a_{k-1}\coloneq \frac{k\varphi(x_k)-\psi(x_{k-1})}{x_k-x_{k-1}}$ is less than $a_{k-2}$,
		\end{itemize}
		the last condition being possible since $\varphi$ is sublinear. Let us define
		$$b_{k-1}\coloneq \psi(x_{k-1})-\frac{k\varphi(x_k)-\psi(x_{k-1})}{x_k-x_{k-1}}x_{k-1}.$$
		We then extend $\psi$ on $[x_{k-1},x_k]$ by setting
		$$\psi(t)\coloneq b_{k-1}+a_{k-1}t=\frac{k\varphi(x_k)-\psi(x_{k-1})}{x_k-x_{k-1}}(t-x_{k-1})+\psi(x_{k-1}),$$
		so that $\psi$ satisfies $\psi(x_k)=k\varphi(x_k)$ and $\psi(t)\leq k\varphi(t)$ for every $t\in [x_{k-1},x_k]$ (by definition of the set $E(x_{k-1},\psi(x_{k-1}),k\times\varphi)$). The real number $b_{k-1}$ is necessarily non-negative since we have $b_{k-2}+a_{k-2}x_{k-1}=b_{k-1}+a_{k-1}x_{k-1}$ with $a_{k-1}<a_{k-2}$ and $b_{k-2}\geq 0$.\par
		Let us prove that $\psi$ satisfies the desired conditions. The map $\psi$ is increasing since the real numbers $a_i$ are positive. It is easy to prove that $\varphi(x_k)=o(\psi(x_k))$. Since the map $t\in (0,+\infty )\mapsto\frac{t}{at+b}$ is non-decreasing if $a>0$ and $b\geq 0$, we get that the map $t\mapsto t/\psi(t)$ is non-decreasing. We build $\psi$ as a concave and increasing map satisfying $\psi(0)=0$, so $\psi$ is subadditive. Finally, given an integer $i\in\{1,\ldots,\ell\}$, we have
		$$\begin{array}{ll}
			\displaystyle \sum_{n=x_1}^{+\infty}{\psi(n)\mu(\{f_i=n\})}&=\displaystyle\sum_{k=1}^{+\infty}{\sum_{n=x_k}^{x_{k+1}-1}{\psi(n)\mu(\{f_i=n\})}}\\
			&\leq \displaystyle \sum_{k=1}^{+\infty}{\sum_{n=x_k}^{x_{k+1}-1}{(k+1)\varphi(n)\mu(\{f_i=n\})}}\\
			&\leq \displaystyle \sum_{k=1}^{+\infty}{\sum_{n=x_k}^{x_{k+1}-1}{K_n\varphi(n)\mu(\{f_i=n\})}}\\
			&= \displaystyle \sum_{n=1}^{+\infty}{K_nu^{(i)}_n}<\infty,
		\end{array}$$
		where the second inequality follows from the inequalities $k+1\leq K_n$ for every integers $n$ and $k$ satisfying $n\geq x_k$ (since we have $x_k\geq N_{k+1}$). The equality
		$$\int_{X}{\psi(f_i(x))\mathrm{d}\mu(x)}=\displaystyle \sum_{n=0}^{x_{1}-1}{\psi(n)\mu(\{f_i=n\})}+\sum_{n=x_1}^{+\infty}{\psi(n)\mu(\{f_i=n\})}$$
		implies that the integral is finite.
	\end{proof}
	
	\begin{proof}[Proof of Theorem~\ref{thA}]
		Suppose that there exist a non-decreasing function $\fg$ and an increasing function $\fl$ satisfying $\fg\approx\jg$, $\fl\approx\jl$ and the following assumptions as $x\to +\infty$:
		\begin{equation}\label{hyp1bis}
			\fg(x)=o\left (\fl(x)\right ),
		\end{equation}
		\begin{equation}\label{hyp2bis}
			\forall C>0,\ \fg(Cx)=O\left (\fg(x)\right ),
		\end{equation}
		\begin{equation}\label{hyp3bis}
			\forall C>0,\ \fg\circ \fl^{-1}(Cx)=O\left (\fg\circ \fl^{-1}(x)\right ).
		\end{equation}
		Let us assume by contradiction that there exists a $(\fg\circ \fl^{-1},\ld^0)$-integrable measure equivalence coupling $(\Omega,\XG,\XL,\mu)$ from $\G$ to $\La$. Let us fix finite generating sets $S_{\G}$ of $\G$ and $S_{\La}$ of $\La$. We write $S_{\G}=\{\g_1,\ldots,\g_{\ell}\}$. For every $i\in\{1,\ldots,\ell\}$, there is a constant $c_{\g_i}>0$ such that
		$$\int_{\XL}{\fg\circ \fl^{-1}\left (\frac{{|\cogl(\g_i,x)|}_{S_{\La}}}{c_{\g_i}}\right )\mathrm{d}\mu_{\XL}(x)}<+\infty.$$
		Using Assumption~\eqref{hyp3bis} for $C=c_{\g_i}$, we may and do assume that $c_{\g_i}=1$ for every $i\in\{1,\ldots,\ell\}$. We now apply Lemma~\ref{lemma2} to $\varphi = \fg\circ \fl^{-1}$ ($\varphi$ is sublinear by Assumption~\eqref{hyp1bis}), $\esp = (\XL,\mu_{\XL})$ and $f_i\colon x\mapsto {|\cogl(\g_i,x)|}_{S_{\La}}$. We thus get that $(\Omega,\XG,\XL,\mu)$ is a $(\psi,\ld^0)$-integrable measure equivalence coupling from $\G$ to $\La$, for some map $\psi\colon\R_+\to\R_+$ satisfying the following properties:
		\begin{enumerate}[label=(\Alph*)]
			\item\label{propA} $\fg\circ \fl^{-1}(x_k)=o(\psi(x_k))$ for some sequence $(x_k)_{k\geq 0}$ of non-negative real numbers tending to $+\infty$;
			\item\label{propB} $\psi$ and $t\mapsto\frac{t}{\psi(t)}$ are non-decreasing;
			\item\label{propC} $\psi$ is subadditive;
		\end{enumerate}
		If we have
		\begin{equation}\label{contradiction}
			\fg\succcurlyeq\psi\circ \fl,
		\end{equation}
		namely $\psi(x)=O\left (\fg(C \fl^{-1}(x))\right )$ for some constant $C>0$, then we get a contradiction with Assumption~\eqref{hyp2bis} and Property~\ref{propA}. Now it remains to prove Inequality~\eqref{contradiction}.\par
		First, Property~\ref{propB} and Theorem~\ref{th1} imply that
		$$\jg\succcurlyeq\psi\circ \jl,$$
		which means that there exist constants $C,D>0$ such that $\psi(\jl(x))\leq D\jg(Cx)$ for every $x\geq 0$. Secondly there also exist constants $C_1, C_2, D_1, D_2>0$ such that $\fl(x)\leq D_1\jl(C_1 x)$ and $\jg(x)\leq D_2\fg(C_2 x)$ for every $x\geq 0$. Moreover, by Property~\ref{propC} and the monotonicity of $\psi$, we have $\psi(cx)\leq\lceil c\rceil\psi(x)$ for every $c>0$. Finally, this gives
		$$\begin{array}{lcl}
			\displaystyle\psi(\fl(x))&\leq&\displaystyle\psi(D_1\jl(C_1 x))\\
			&\leq&\displaystyle\lceil D_1\rceil\psi(\jl(C_1 x))\\
			&\leq&\displaystyle\lceil D_1\rceil D\jg(CC_1 x)\\
			&\leq&\displaystyle\lceil D_1\rceil DD_2\fg(CC_1 C_2 x)
		\end{array}$$
		and we get Inequality~\eqref{contradiction}.
	\end{proof}
	
	\begin{proof}[Proof of Theorem~\ref{thB}]
		This is the same proof as Theorem~\ref{thA}, except that we get a contradiction with Theorem~\ref{th2}, using the fact that Lemma~\ref{lemma2} yields a map $\psi$ which can be increasing and subadditive and satisfy $\psi(0)=0$. Moreover we similarly prove that $V_{\G}\succcurlyeq V_{\La}\circ\psi^{-1}$ implies $\fg\succcurlyeq \fl\circ\psi^{-1}$.
	\end{proof}
	
	\section{Applications}\label{secapp}
	
	\subsection{Coupling from a finitely generated group to \texorpdfstring{$\Z$}{Z}}\label{secappZ}
	
	\begin{corollary}\label{cor3}
		Let $\G$ be a finitely generated group which is not virtually cyclic. Assume that its isoperimetric profile $\jg$ satisfies
		\begin{equation}\label{hyp7}
			\forall C>0,\ \jg(Cx)=O\left (\jg(x)\right )\text{ as }x\to +\infty.
		\end{equation}
		Then there is no $(\jg,\ld^0)$-integrable measure equivalence coupling from $\G$ to $\Z$.
	\end{corollary}
	
	\begin{proof}[Proof of Corollary~\ref{cor3}]
		A group $\Gamma$ is not virtually cyclic if and only if $\jg(x)=o(x)$. This is a consequence of the Coulhon Saloff-Coste isoperimetric inequality~\cite[Theorem~1]{coulhonIsoperimetriePourGroupes1993} and the fact that the volume growth of such a group is at least quadratic if it is not virtually cyclic (see e.g.~\cite[Corollary~3.5]{mannHowGroupsGrow2011a}). We then apply Theorem~\ref{thA} and Remark~\ref{isoinj} to get Corollary~\ref{cor3}.
	\end{proof}
	
	In~\cite[Theorem~1.1]{brieusselSpeedRandomWalks2021} Brieussel and Zheng prove that for any non-decreasing function $f\colon\R_+\to\R_+$ such that $x\mapsto x/f(x)$ is non-decreasing, there exists a group $\G$ such that $\jg\approx \frac{\log}{f\circ\log}$, we call it a Brieussel-Zheng group (although their construction is more general).\par
	Defining the map $F\coloneq \frac{\log}{f\circ\log}$,
	the monotonicity of $f$ (resp. of $x\mapsto x/f(x)$) implies that $F/\log$ is non-increasing (resp. $F$ is non-decreasing) and the converse is true. Therefore, any non-decreasing function $F\colon [1,\infty)\to [1,\infty)$ such that $F/\log$ is non-increasing is the isoperimetric profile of a group. This equivalent statement was already noticed in \cite[Theorem~4.26]{delabieQuantitativeMeasureEquivalence2022}.\par
	From this we deduce that the isoperimetric profiles provided by Brieussel and Zheng satisfy Assumption~\eqref{hyp7}. Indeed, let $F$ be a non-decreasing function such that $F/\log$ is non-increasing, and let $C$ be a positive constant. If $C<1$, then the monotonicity of $F$ directly implies the inequality $F(Ct)\leq F(t)$. If $C\geq 1$, we get
	$$\frac{F(Cx)}{\log{(Cx)}}\leq \frac{F(x)}{\log{(x)}}$$
	by monotonicity of $F/\log$, so we have $F(Cx)\leq F(x)\frac{\log{(Cx)}}{\log{(x)}}$, where the right-hand side is less than $2F(x)$ when $x$ is large enough.\par
	As mentionned in the introduction, Escalier~\cite[Theorem~1.7]{escalierBuildingPrescribedQuantitative2024} proves that for every\footnote{Actually, the statement of Theorem~1.7 in~\cite{escalierBuildingPrescribedQuantitative2024} is the following : given a non-decreasing function $F$ such that $F/\log$ is non-decreasing, there exists a group $\G$ such that $\jg\approx F$ and there exists an orbit equivalence coupling from $\G$ to $\Z$ which is $(\varphi_{\varepsilon},\exp\circ F\circ\exp)$ for every $\varepsilon>0$, where $\varphi_{\varepsilon}(x)=F(x)/(\log{F(x)})^{1+\varepsilon}$. The group $\G$ is in fact a Brieussel-Zheng group and the proof of the theorem shows that the existence of such an orbit equivalence holds for every such groups.} Brieussel-Zheng group $\G$ mentionned above, there exists an orbit equivalence coupling from $\G$ to $\Z$ which is $(\varphi_{\varepsilon},\ld^0)$-integrable for all $\varepsilon>0$, where $\varphi_{\varepsilon}(x)=\frac{\jg(x)}{(\log{\jg(x)})^{1+\varepsilon}}$. Hence, we deduce the following.
	
	\begin{theorem}\label{BrieusselZhengEscalier}
		Let $\G$ be a Brieussel-Zheng group and $p>0$. Then there exists a $((\jg)^p,\ld^0)$-integrable measure equivalence from $\G$ to $\Z$ if and only if $p<1$.
	\end{theorem}
	
	\subsection{Coupling between groups of polynomial growth}
	
	\begin{corollary}\label{cor4}
		Assume that $\G$ and $\La$ have polynomial growth of degree $b$ and $a$ respectively, with $b>a$. Then there is no $(\ld^{a/b},\ld^0)$ measure equivalence coupling from $\G$ to $\La$.
	\end{corollary}
	
	\begin{proof}[Proof of Corollary~\ref{cor4}]
		The isoperimetric profiles satisfy $\jg(x)\approx x^{1/b}$ and $\jl(x)\approx x^{1/a}$ (see~\cite[Theorem~1]{coulhonIsoperimetriePourGroupes1993}), so the corollary follows from Theorem~\ref{thA}.
	\end{proof}
	
	As mentionned in the introduction, Delabie, Koivisto, Le Maître and Tessera~\cite{delabieQuantitativeMeasureEquivalence2022} explicitly build an orbit equivalence in the special case of the groups $\Z^d$ for $d\geq 1$, and then show that there exists a measure equivalence coupling from $\Z^b$ to $\Z^a$ (with $b>a$) which is $(\ld^p,\ld^0)$-integrable for every $p<a/b$. But the existence of a $(\ld^{a/b},\ld^0)$-integrable coupling remained unclear. Our Corollary~\ref{cor4} then gives the following complete description.
	
	\begin{theorem}\label{th3}
		Given positive integers $b>a$, there exists a $(\ld^p,\ld^0)$ measure equivalence coupling from $\Z^b$ to $\Z^a$ if and only if $p<a/b$.
	\end{theorem}
	
	\subsection{Lamplighter groups}\label{secappwr}
	
	Let $G$ and $F$ be two countable groups and $\bigoplus_{g\in G}{F}$ be the subgroup of $F^G$ consisting of all functions with finite support\footnote{The support of a function $f\colon G\to F$ is the set $\{g\in G\mid f(g)\not =e_F\}$ where $e_F$ is the identity element of $F$.}. We define the action of $G$ on $\bigoplus_{g\in G}{F}$ as follows. For every $g\in G$ and every $f\in\bigoplus_{g\in G}{F}$, the function $g\cdot f\in\bigoplus_{g\in G}{F}$ is defined by:
	$$\forall g'\in G,\ (g\cdot f)(g')= f(g^{-1}g').$$
	Then the \textit{wreath product} $F\wr G$ is the semi-direct product
	$$F\wr G\coloneq \left (\bigoplus_{g\in G}{F}\right )\rtimes G.$$
	When $F$ is a non-trivial finite group, $F\wr G$ is also called a \textit{lamplighter group}.
	
	\begin{corollary}\label{cor5}
		Assume that $G$ and $H$ have polynomial growth of degree $b$ and $a$ respectively, with $b>a$, and let $F$ and $K$ be non-trivial finite groups. Then there is no $(\ld^{a/b},\ld^0)$ measure equivalence coupling from $F\wr G$ to $K\wr H$.
	\end{corollary}
	
	\begin{proof}[Proof of Corollary~\ref{cor5}]
		The isoperimetric profiles satisfy $j_{1,F\wr G}(x)\approx (\log{x})^{1/b}$ and $j_{1,K\wr H}(x)\approx (\log{x})^{1/a}$ (see \cite[Theorem~1]{erschlerIsoperimetricProfilesFinitely2003}), so the corollary follows from Theorem~\ref{thA}.
	\end{proof}
	
	In the case $F=K$, $G=\Z^b$ and $H=\Z^a$, using the notion of wreath product for measure-preserving equivalence relations, Corollary~7.4 in~\cite{delabieQuantitativeMeasureEquivalence2022} implies that there exists a $(\ld^p,\ld^0)$ measure equivalence coupling from $F\wr\Z^b$ to $F\wr\Z^a$ for every $p<a/b$. Combined with Corollary~\ref{cor5}, this yields the following theorem.
	
	\begin{theorem}\label{th4}
		Given positive integers $b>a$, there exists a $(\ld^p,\ld^0)$ measure equivalence coupling from $F\wr\Z^b$ to $F\wr\Z^a$ if and only if $p<a/b$.
	\end{theorem}
	
	\begin{corollary}\label{cor6}
		Assume that $G$ and $\La$ have polynomial growth of degree $b$ and $a$ respectively, with $b>a$, and let $F$ be a non-trivial finite group. Then there is no $(\log^{1/b},\ld^0)$-integrable measure equivalence coupling from $F\wr G$ to $\La$.
	\end{corollary}
	
	\begin{proof}[Proof of Corollary~\ref{cor6}]
		The isoperimetric profiles satisfy $j_{1,F\wr G}(x)\approx (\log{x})^{1/b}$ and $j_{1,\La}(x)\approx x^{1/a}$ (see \cite[Theorem~1]{erschlerIsoperimetricProfilesFinitely2003} and~\cite[Theorem~1]{coulhonIsoperimetriePourGroupes1993}), so we are done by Theorem~\ref{thA}.
	\end{proof}
	
	In the case $G=\Z$ and $\La=\Z$, it is shown in~\cite[Proposition~6.20]{delabieQuantitativeMeasureEquivalence2022} that there exists a $(\log^p,\ld^0)$-integrable measure equivalence coupling from $F\wr\Z$ to $\Z$ for every $p<1$ (this statement deals with $F=\Z/m\Z$ but remains true for any finite group), and Corollary~\ref{cor6} completes this result.
	
	\begin{theorem}\label{th5}
		Given a finite group $F$, there exists a $(\log^p,\ld^0)$-integrable measure equivalence coupling from $F\wr\Z$ to $\Z$ if and only if $p<1$.
	\end{theorem}
	
	\subsection{Iterated wreath products}\label{secappiterwr}
	
	Given an integer $k\geq 1$ and a finite group $F$, we define groups $H_n(k)$ inductively as follows: $H_0(k)=\Z^k$ and $H_{n+1}(k)=F\wr H_n(k)$. Given a positive integer $n$, the map $\log^{\circ n}$ denotes the composition $\log\circ\ldots\circ\log$ ($n$ times).
	
	\begin{corollary}\label{cor7}
		\begin{itemize}
			\item If $b>a$, then there is no $(\ld^{a/b},\ld^0)$ measure equivalence coupling from $H_n(b)$ to $H_n(a)$.
			\item Given integers $d,k\geq 1$, there is no $((\log^{\circ n})^{1/k},\ld^0)$-integrable measure equivalence coupling from $H_n(k)$ to $\Z^d$.
		\end{itemize}
	\end{corollary}
	
	\begin{proof}[Proof of Corollary~\ref{cor7}]
		The isoperimetric profiles satisfy $j_{1,H_n(k)}(x)\approx (\log^{\circ n}{x})^{1/k}$ (see \cite[Theorem~1]{erschlerIsoperimetricProfilesFinitely2003}), and $j_{1,\Z^d}(x)\approx x^{1/d}$. Then the corollary follows from Theorem~\ref{thA}.
	\end{proof}
	
	Using the notion of wreath products of measure-preserving equivalence relations, it is proven in~\cite[Corollary~7.5]{delabieQuantitativeMeasureEquivalence2022} that there exists a $(\ld^p,\ld^0)$ measure equivalence coupling from $H_n(b)$ to $H_n(a)$ for every $p<a/b$. Moreover the composition of couplings yields a $((\log^{\circ n})^{p},\ld^0)$ measure equivalence coupling from $H_n(1)$ to $\Z$ for every $p<1$ (see~\cite[Corollary~7.6]{delabieQuantitativeMeasureEquivalence2022}). Our results allow us to complete these observations.
	
	\begin{theorem}\label{th6}
		Given positive integers $b>a$, there exists a $(\ld^p,\ld^0)$ measure equivalence coupling from $H_n(b)$ to $H_n(a)$ if and only if $p<a/b$.
	\end{theorem}
	
	\begin{theorem}\label{th7}
		Given integers $d,k\geq 1$, there exists a $((\log^{\circ n})^p,\ld^0)$-integrable measure equivalence coupling from $H_n(1)$ to $\Z$ if and only if $p<1$.
	\end{theorem}
	
	\begin{remark}
		All the measure equivalence couplings provided in~\cite{escalierBuildingPrescribedQuantitative2024} and~\cite{delabieQuantitativeMeasureEquivalence2022} and that we have mentioned in Section~\ref{secapp} actually come from a construction of orbit equivalences between the groups, with the same integrability for the cocycles. Then Theorems~\ref{th3},~\ref{th4},~\ref{th5},~\ref{th6} and~\ref{th7} remain valid in the context of quantitative orbit equivalence.
	\end{remark}
	
	\bibliographystyle{alphaurl}
	\bibliography{biblio}
	
	{\bigskip
		\footnotesize
		
		\noindent C.~Correia, \textsc{Université Paris Cité, Institut de Mathématiques de Jussieu-Paris Rive Gauche, 75013 Paris, France}\par\nopagebreak\noindent
		\textit{E-mail address: }\texttt{corentin.correia@imj-prg.fr}}
	
\end{document}